\documentclass{amsart}

\usepackage{latexsym, amssymb, amsmath, mathrsfs}
\newtheorem{theorem}{Theorem}[section]
\newtheorem{lemma}[theorem]{Lemma}

\newtheorem{question}[theorem]{Question}

\newtheorem{corollary}[theorem]{Corollary}
\newtheorem{proposition}[theorem]{Proposition}

\newtheorem{condition}[theorem]{Condition}

\theoremstyle{definition}
\newtheorem{definition}[theorem]{Definition}

\newtheorem{previous assumption}[theorem]{Previous Assumption}

\theoremstyle{remark}

\numberwithin{equation}{section}
\newcommand{\pairwith}[2]{\langle{#1},{#2}\rangle}

\newcommand{\GL}[1]{\mathrm{GL}_{#1}}

\newcommand {\EcurlCpps}[1] {E_{\curlC_{#1}}(q,q^{-s})}

\newcommand {\funddom}[1] {D_{#1}}
\newcommand {\valuation}[1] {v(#1)}
\newcommand {\loclinpiece}[1] {\loclinfnonZlattice(\cellassignfn_{#1})}

\def \alphapairxi {\ensuremath{\pairwith{\alpha}{\xi}}}
\def \Aut {\mathrm{Aut}}

\def \begindisplayarray {\begin{displaymath}\begin{array}}
\def \begindm {\begin{displaymath}}

\def \boldanought {\ensuremath{\boldsymbol{a}}}
\def \bolde {\ensuremath{\boldsymbol{e}}}

\def \boldnought {\ensuremath{\boldsymbol{0}}}

\def \boldt {\ensuremath{\boldsymbol{t}}}
\def \boldu {\ensuremath{\boldsymbol{u}}}
\def \boldv {\ensuremath{\boldsymbol{v}}}

\def \boldx {\ensuremath{\boldsymbol{x}}}
\def \boldy {\ensuremath{\boldsymbol{y}}}

\def \cellassignfn {F}
\def \cellsofcomplex {\curlF}

\def \ckji {\ensuremath{c_k(j,i)}}

\def \curlB {\mathcal{B}}
\def \curlC {\mathcal{C}}

\def \curlF {\mathcal{F}}
\def \curlFext {\mathcal{F}^2}

\def \curlG {\mathcal{G}}

\def \curlL {\mathcal{L}}

\def \curlU {\mathcal{U}}

\def \curlV {\mathcal{V}}

\def \curlW {\mathcal{W}}

\def \detrho {\ensuremath{\mathrm{det}\rho}}
\def \detrhog {\ensuremath{\mathrm{det}\rho(g)}}

\def \detrhoxipi {\ensuremath{\textrm{det}\rho(\xi(\pi))}}

\def \diag {\ensuremath{\textrm{diag}}}
\def \disjointunion {\ensuremath{\coprod}}
\def \enddisplayarray {\end{array}\end{displaymath}}
\def \enddm {\end{displaymath}}

\def \fnew {f}

\def \Fp {\mathbb{F}_p}

\def \Gadd {\mathbb{G}_a}
\def \genfn {E}
\def \Gmult {\mathbb{G}_m}
\def \GLn {\mathrm{GL}_n}

\def \globalfield {{\mathfrak{K}}}
\def \globalfielda {\mathfrak{K}}
\def \globalringofintegers {\mathfrak{O}}

\def \gothicp {\ensuremath{\mathfrak{p}}}

\def \HomGmT {\ensuremath{\mathrm{Hom}(\mathbb{G}_m,T)}}
\def \HomTGm {\ensuremath{\mathrm{Hom}(T,\mathbb{G}_m)}}

\def \idealp {\mathfrak{p}}

\def \intcurlGplus {\ensuremath{\int_{\mathcal{G}^{+}}}}

\def \intGplus {\ensuremath{\int_{G^{+}}}}

\def \integers {\mathbb{Z}}
\def \integerslplusd {\mathbb{Z}^{l+d}}

\def \localfield {{\mathfrak{K}_{\mathfrak{p}}}}
\def \localringofintegers  {\mathfrak{o}}
\def \localringofintegersa {\mathfrak{o}}
\def \loclinfnonZlattice {\gamma}

\def \matrixstart {\left( \begin{array}}
\def \matrixend {\end{array} \right) }
\def \maxideal {\mathfrak{p}}
\def \Mn {\ensuremath{\mathrm{M}_n}}

\def \muG {\ensuremath{\mu_G}}

\def \naturals {\mathbb{N}}
\def \naturalszero {\mathbb{N}_0}

\def \nonredG {\mathcal{G}}
\def \normal {\ensuremath{\lhd}}

\def \negroots {\ensuremath{\Phi^-}}

\def \pgoestopinverse {\ensuremath{q\rightarrow q^{-1}}}

\def \posroots {\ensuremath{\Phi^+}}
\def \Qpadic {\mathbb{Q}_p}
\def \Qp {\mathbb{Q}_p}

\def \rationals {\mathbb{Q}}
\def \reals {\mathbb{R}}

\def \realsnonneg {\mathbb{R}_{\geq0}}

\def \thespan {\mbox{span}}

\def \suchthat {\ensuremath{\,|\,}}

\def \Sym {\ensuremath{\textrm{Sym}}}
\def \tgroup {\mathscr{T}}
\def \tensor {\otimes}

\def \thetaxipi {\ensuremath{\theta(\xi(\pi))}}

\def \unifparam {\pi}

\def \weightsi {\tilde{\omega}}
\def \wtxi {\ensuremath{wt_\xi}}
\def \wxi {\ensuremath{w_\xi}}
\def \wXiw {\ensuremath{w\Xi_w}}
\def \wXiwplus {\ensuremath{w\Xi_w^+}}
\def \wxitxi {\ensuremath{w_\xi t_\xi}}

\def \Xiplus {\ensuremath{\Xi^+}}
\def \Xiw {\ensuremath{\Xi_w}}
\def \Xiwplus {\ensuremath{\Xi_w^+}}
\def \xipi {\ensuremath{\xi(\pi)}}
\def \xibasisi {\ensuremath{\xi}}
\def \xibasis {\ensuremath{\underline{\xi}}}

\def \Zp  {\mathbb{Z}_p}

\def \ZnonredGrhops {\ensuremath{Z_{\mathcal{G},\rho,{\maxideal}}(s)}}
\def \ZcurlGrhops {\ensuremath{Z_{\nonredG,\rho,\maxideal}(s)}}
\def \ZGrhops {\ensuremath{Z_{G,\rho,\theta,\maxideal}(s)}}
\def \ZcurlGrhoidealps {\ensuremath{Z_{\mathcal{G},\rho,\idealp}(s)}}

\begin{document}
\pagestyle{plain}

\title{Uniformity and Functional Equations for Local Zeta Functions of $\mathfrak{K}$-split Algebraic Groups}
\author{By Mark N. Berman}

\subjclass[2000]{05A15, 11S45, 20E07, 22E50}

\keywords{Zeta functions of algebraic groups, zeta functions of finitely generated torsion-free nilpotent groups, enumerative combinatorics, $p$-adic integration, uniformity, local functional equations}


\begin{abstract}
We study the local zeta functions of an algebraic group $\mathcal{G}$ defined over $\globalfield$ together with a faithful $\globalfield$-rational representation $\rho$ for a finite extension $\globalfield$ of $\mathbb{Q}$. These are given by integrals over $\maxideal$-adic points of $\mathcal{G}$ determined by $\rho$ for a prime $\maxideal$ of $\globalfield$. We prove that the local zeta functions are almost uniform for all $\globalfield$-split groups whose unipotent radical satisfies a certain lifting property. This property is automatically satisfied if $\mathcal{G}$ is reductive. We provide a further criterion in terms of invariants of $\mathcal{G}$ and $\rho$ which guarantees that the local zeta functions satisfy functional equations for almost all primes of $\globalfield$. We obtain these results by using a $\maxideal$-adic Bruhat decomposition of Iwahori and Matsumoto \cite{IwMa} to express the zeta function as a weighted sum over the Weyl group $W$ associated to $\mathcal{G}$ of generating functions over lattice points of a polyhedral cone. The functional equation reflects an interplay between symmetries of the Weyl group and reciprocities of the combinatorial object. We construct families of groups with representations violating our second structural criterion whose local zeta functions do not satisfy functional equations. Our work generalizes results of Igusa \cite{Igusa} and du Sautoy and Lubotzky \cite{duSL} and has implications for zeta functions of finitely generated torsion-free nilpotent groups.

\end{abstract}

\maketitle

\section{Introduction}\label{introduction section} Let $\curlG$ be an algebraic group defined over a finite extension $\globalfield$ of $\rationals$ and
let $\nobreak{\rho:\nonredG\to \GL{n}}$ be a faithful $\globalfielda$-rational
representation of $\nonredG$. Let $\globalringofintegers$ be the ring of integers of $\globalfield$ and $\idealp$ a prime of $\globalringofintegers$. We denote the localization of $\globalfield$ at $\maxideal$ by $\localfield$, its ring of integers by $\localringofintegers$ (the dependence on $\maxideal$ being understood) and the size of its residue field by $q$. Let $\unifparam$ be a
fixed uniformizing parameter for $\localringofintegers$. For $X\leq \nonredG$ put
\begin{eqnarray*}
X^+&:=&\{g\in X(\localfield)\suchthat\rho(g)\in\Mn(\localringofintegersa)\},\\
X(\localringofintegersa)&:=&\{g\in X(\localfield)\suchthat\rho(g)\in\GL{n}(\localringofintegersa)\}.
\end{eqnarray*}
We define the local zeta function of $(\nonredG,\rho)$ at the prime
$\idealp$ to be
\begindm
\ZcurlGrhoidealps:=\intcurlGplus|\det{\rho(g)}|^s\mu_{\nonredG}(g),
\enddm
where $s$ is a complex variable, $|.|$ is the $\idealp$-adic absolute value and $\mu_{\nonredG}$ is
the right invariant Haar measure on $\nonredG(\globalfielda_\idealp)$ normalized such that
$\mu_{\nonredG}(\nonredG(\localringofintegersa))=1$.

The local zeta functions $\ZcurlGrhoidealps$ are said to be (almost) uniform if there exists a rational function $Q(X,Y)$ such that for (almost) all primes $\maxideal$, $\ZcurlGrhoidealps=Q(q,q^{-s})$. In this case, set $\ZcurlGrhoidealps|_{q\to q^{-1}}:=Q(q^{-1}, q^s)$. We say that $\ZcurlGrhoidealps$ satisfies a functional equation if
\begin{equation}\label{equation fn eq}
\ZcurlGrhoidealps|_{q\to q^{-1}}=(-1)^{m}q^{a+bs}\ZcurlGrhoidealps
\end{equation}
for some integers $m,a,b$. In fact there is a well-defined notion of a functional equation even in cases in which the local zeta functions attached to an object are not almost uniform but rather are determined by counting rational points on certain varieties (see, for instance, \cite{vollfneq}). We will impose conditions on $(\nonredG,\rho)$ that will imply that the local zeta functions are almost uniform. However, the question of uniformity for these integrals, in general, remains open. 

Early interest in the function $\ZcurlGrhoidealps$ came from the fact that it is a natural generalization of the Dedekind zeta function. Indeed, the latter may be obtained by taking $\nonredG=\GL{1}$ and $\rho$ the natural representation. It was studied in other special cases by Hey, Weil, Tamagawa, Satake and Macdonald \cite{hey, weil, tamagawa, satake, macdonald}. Tamagawa considered the case $\nonredG=\GL{n}$ with the natural representation and showed that the global zeta function (defined to be the product of the local zeta functions over all primes $\maxideal$) has meromorphic continuation to the whole of the complex plane. In \cite[Chapter 6]{duSW} the authors proved that for several families of classical groups, the global zeta functions have natural boundaries and thus cannot be meromorphically continued. Nevertheless, it remains interesting to ask which properties of the Dedekind zeta function carry over to the function $\ZcurlGrhoidealps$ for all groups $\nonredG$ and representations $\rho$.

The question of whether the functions $\ZcurlGrhoidealps$ are (almost) uniform and satisfy functional equations was first addressed in a more general setting by Igusa \cite{Igusa}. He chose Serre's canonical measure on $\nonredG$, which differs from the Haar measure used in our definition. He was able to derive an explicit form for $\ZcurlGrhoidealps$ in terms of $q$ and $q^{-s}$. This form involved a certain sum, over the Weyl group $W$ of $\nonredG$, of rational functions weighted by the length function on $W$. Igusa was able to utilize a symmetry of the Weyl group together with reciprocities for the rational functions to prove that the local zeta functions satisfy functional equations for almost all primes. The main tool he employed was a $\gothicp$-adic Bruhat decomposition due to Iwahori and Matsumoto \cite{IwMa}. Although his integral differed from our ours, the method he employed will be essential to our analysis.

Independent interest in the function $\ZcurlGrhops$ was generated by \cite{GSS}, in which the authors showed that the zeta function in fact expresses a subgroup counting function in a completely different context. Let $\Gamma$ be a finitely generated torsion-free nilpotent group (or $\tgroup$-group) and for $*\in\{\leq,\normal,\wedge\}$ define 
\begin{displaymath}
\displaystyle{\zeta^{*}_{\Gamma}(s):=\sum_{n=1}^{\infty}a^{*}_n n^{-s},}
\end{displaymath}
where $a^{\leq}_n, a^{\normal}_n$ and $a^{\wedge}_n$ denote the number of subgroups $H$ of $\Gamma$ of index $n$ satisfying $H\leq\Gamma$, $H\normal\Gamma$ and $\widehat{H}\cong\widehat{\Gamma}$ respectively . Here $\widehat{\phantom{\Gamma}}$ denotes the profinite completion. $\zeta^{\wedge}_{\Gamma}(s)$ has been given no name to date; we refer to it as the proisomorphic zeta function of $\Gamma$. Define local zeta functions for each prime $p$ by
\begindm
\displaystyle{\zeta^{*}_{\Gamma,p}(s):=\sum_{k=0}^{\infty}a^{*}_{p^k} p^{-ks}.}
\enddm
Grunewald, Segal and Smith showed that for $*\in\{\leq,\normal,\wedge\}$ and for each prime $p$, $\zeta_{\Gamma, p}^*(s)$ is a rational function in $p^{-s}$. They also showed that, as a straightforward consequence of nilpotency, the zeta function decomposes as an Euler product:
\begindm
\displaystyle{\zeta^*_{\Gamma}(s)=\prod_{p}\zeta^*_{\Gamma, p}(s).}
\enddm
Furthermore, they realized the local proisomorphic zeta function $\zeta^{\wedge}_{\Gamma,p}(s)$ as the local zeta function $Z_{\curlG, \rho, p}(s)$ of an algebraic group with an associated $\rationals$-rational representation (hence all such realizations are \emph{a fortiori} rational functions in~$p^{-s}$). This prompted du Sautoy and Lubotzky to study $\ZcurlGrhoidealps$ with a view to ascertaining whether the local proisomorphic zeta functions $\zeta^{\wedge}_{\Gamma,p}(s)$ would be uniform and whether they would satisfy functional equations. 

It is well-known that if $\nonredG_0$ is the connected component of $\nonredG$ then $\nonredG_0$ can be expressed as $N\rtimes G$, where $N$ is the unipotent radical of $\nonredG_0$ and $G$ is a (connected) reductive subgroup (see, for instance, \cite[p.\ 9]{Borel}). In \cite{duSL} du Sautoy and Lubotzky made a reduction to an integral over the subgroup $G$ and showed that the zeta function $\ZcurlGrhops$ would be unchanged for almost all primes $\maxideal$. Specifically, they showed that
\begin{equation}\label{equation duSL reduction}
\ZnonredGrhops=\int_{G^+}|\detrhog|^s\theta(g)\mu_{G}(g),
\end{equation} where $\theta$ is the function $G\to\realsnonneg$ given by
\begin{equation}\label{definition first of theta}
\theta(g):=\mu_N(\{n\in N\suchthat ng\in\nonredG^+_0\})
\end{equation}
and $\mu_N$ is the right invariant Haar measure on $N$ normalized such that ${\mu_N(N(\localringofintegers))=1}$. They were also able to decompose the function $\theta$ into pieces defined relative to a normal series for $N$; this relied on a certain lifting assumption on $\nonredG_0$ (see \cite[Assumption~2.3]{duSL}, restated in our paper as Condition~\ref{lifting condition}). They further assumed that $G$ would split over $\globalfield$, that $\theta$ would be (the $\maxideal$-adic absolute value of) a character on $G$, that the rank of the maximal central torus of $G$ would be $1$, and that among the irreducible components $\rho_1,\ldots,\rho_r$ of $\rho|_G$ there would be one whose dominant weight `dominates' the dominant weights of the other components. We refer the reader to \cite{duSL} for a definition of the latter. Under these assumptions, they were able to extend Igusa's method to obtain an explicit form for the local zeta functions $\ZcurlGrhops$, deducing that they would be almost uniform in $\maxideal$ and would satisfy functional equations.

We now state our main theorem.

\begin{theorem}~\label{main theorem}
Let $\nonredG$ be an algebraic group defined over $\globalfield$ and
let $\rho$ be a faithful $\globalfield$-rational representation $\nonredG\to\GL{n}$. Write $\nonredG_0=N \rtimes G$, where $\nonredG_0$ is
the connected component of $\nonredG$, $N$ is its unipotent radical
and $G$ is a connected reductive subgroup. Suppose that, after fixing some
suitable decomposition of the representation space, $(\nonredG_0,
G,\rho)$ satisfies Condition~{\ref{lifting condition}}
(cf. Section~\ref{reduction section}) for almost all primes $\maxideal$. Suppose
further that $G$ splits over
$\globalfield$. Let $d$ denote the rank of the maximal central torus of $\nonredG_0$ and $r$ the number of
irreducible components of $\rho|_G$. Then
\begin{enumerate}
\item the local zeta functions $\ZnonredGrhops$ are almost uniform, i.e. uniform outside a finite set of primes
\item if $r=d$, for almost all primes $\maxideal$, $\ZnonredGrhops$ satisfies the functional
equation
\begindm
\ZnonredGrhops|_{q\to q^{-1}}=(-1)^{l+d}q^{|\Phi^+|+c-ns}\ZnonredGrhops
\enddm
where $l$ is the number of fundamental roots of the root system associated to~$G$, $\Phi^+$ is a set of positive roots, and $c$ is a non-negative integer. If $\nonredG$ is reductive then $c=0$.
\end{enumerate}
\end{theorem}

For a faithful representation we necessarily have $r\geq d$ (each irreducible component has a maximal central torus of rank $1$ or $0$; see \cite[p.\ 700]{Igusa}). We will show in Section~\ref{counterexample section} that in one direction our result is best possible in the following sense: there exist groups $\nonredG$ and representations $\rho$ for which $r>d$ whose local zeta functions do not satisfy functional equations in the sense of (\ref{equation fn eq}). A more complicated counterexample has in fact been given previously by Martin \cite{Martin}. He studied the integral as defined by Igusa for a certain $3402$-dimensional irreducible representation of $\GL{7}$. Here, since $r=d=1$, we infer from our Theorem~\ref{main theorem} that it is the different choice of measure in this example (namely the canonical measure) that is responsible for the break-down of the functional equation.

The proof of Theorem~\ref{main theorem} relies on the splitting assumption to utilize a $\maxideal$-adic Bruhat decomposition of Iwahori-Matsumoto. As suggested in \cite[p.\ 73]{duSL}, it may be possible to remove this assumption using the notion of a $\maxideal$-adic decomposition for non-split reductive groups due to Bruhat and Tits (see \cite{bruhattitsI} and \cite{bruhattitsII}). Condition~{\ref{lifting condition}} is needed to reduce to the case of a reductive group, and cannot easily be removed without a deeper understanding of the action of the reductive subgroup on the unipotent radical, as described below. In the case that $\nonredG$ is reductive, Condition~\ref{lifting condition} holds trivially. We thus obtain the following.

\begin{corollary}
If $\nonredG$ is a $\globalfield$-split reductive group defined over $\globalfield$ with faithful $\globalfield$-rational representation $\rho$ then its local zeta functions satisfy statements (1) and (2) of Theorem~\ref{main theorem}.
\end{corollary}
Our theorem generalizes \cite[Theorem~6.1]{duSL} as follows. We allow the maximal central torus to have arbitrary rank and we no longer make the assumption that the function $\theta$ is (the $\maxideal$-adic absolute value of) a
character on $G$. We do not assume that there is any relationship between the dominant weights of the irreducible components of $\rho|_G$; rather, we restrict the number of components to $d$, the rank of the maximal central torus of $G$. In \cite{duSL}, the authors proved a functional equation in the case that $d=1$, $r$ is arbitrary and among the dominant weights of the irreducible components of the representation, there is one which `dominates' all the others. Note that our Theorem~\ref{main theorem}~(2) does not reduce to their Theorem~A~(2) under their assumptions.  In fact their proof of the latter is flawed, which we will explain at the end of Section~\ref{functional equation subsection}. However, in the case $r=1$ their proof is valid and their result becomes a special case of our Theorem~\ref{main theorem}~(2).

The functional equation in part (2) of our main theorem comes, as in Igusa's setting, from an interplay between symmetries of the Weyl group associated to the algebraic group, and reciprocities of the generating functions in the weighted sum. However, in our case these generating functions are not just geometric series (as in \cite{Igusa} and \cite{duSL}); rather, they are certain generalized generating functions over subsets of a polyhedral cone. The reciprocity property we use to prove the functional equation is an extension of a reciprocity theorem due to Stanley \cite[Theorem~4.6.14]{Stanley}. Our proof relies crucially on the fact that the cone is simple. By this we mean that the set of lattice points it contains is freely generated as a commutative monoid (note that a simple cone is, in particular, simplicial). The simplicity of the cone is in turn a consequence of the condition $r=d$. Indeed, it is this insight that enables us to construct our counterexamples in Section~\ref{counterexample section}. An important feature of our combinatorial data is the existence of a `minimal vector' in Corollary~\ref{fn eq corollary} below, in a similar vein to \cite[Corollary 4.6.16]{Stanley}, although our case is slightly different due to the presence of internal structure in the cone. Our proof of a functional equation is reminiscent of \cite{vollfneq}, where the generating functions varied in different cells defined by linear forms (see \cite[Proposition~2.1]{vollfneq}). In \cite[Theorem~A]{vollklopsch2}, the authors considered sums of generating functions possessing `inversion properties' analogous to Lemma~\ref{inversion of a cell lemma} in the present paper, weighted by the numbers of non-degenerate flags in a finite formed space. While the settings are all different, they share a common feature in utilizing both combinatorial reciprocity properties and symmetries of a Weyl or Coxeter group to prove a functional equation (as is the case also in \cite{Igusa} and \cite{duSL}).

We now explain why we have included Condition~\ref{lifting condition} in Theorem~\ref{main theorem} and how this relates to \cite[Theorem~6.1]{duSL}. In \cite{duSL}, the authors analyzed the automorphism group of $U^0_4$, the Lie algebra of upper triangular 4 by 4 matrices, together with a natural representation, and showed that in this case $\theta$ (defined in (\ref{definition first of theta})) is only a `piecewise' character. Specifically, they were able to divide the domain of integration into two regions such that $\theta$ was a character on each. In Section~\ref{reduction section} we will make this notion of `piecewise characters' more precise and show that $\theta$ will be a piecewise character provided that $({\nonredG_0},\rho)$ satisfies Condition~\ref{lifting condition}. We need this property of $\theta$ in order to carry out our analysis of the reduced integral expression for $\ZnonredGrhops$ given in (\ref{equation duSL reduction}).
Since the condition is somewhat technical, we postpone a proper description of it until Section~\ref{reduction section}. Roughly speaking, the condition states that quotients of ${\nonredG_0}$ by certain normal subgroups $N_i$ of the unipotent radical, together with natural representations $\varphi_i:{\nonredG_0}/N_i\to \GL{n}$, enjoy the property that integral points of $\nonredG_0/N_i$ defined with respect to $\varphi_i$ lift to integral points of ${\nonredG_0}$ with respect to~$\rho$. Corollary~4.5 in \cite{duSL} states that our Condition~\ref{lifting condition} is
satisfied for almost all~$\maxideal$. This is in fact incorrect. For instance, if $\nonredG$ is taken to be the
automorphism group of $U_{5}^{0}$ and
$\rho$ the representation with respect to the standard basis for
$U_{5}^{0}$, it can be checked by a straightforward calculation that Condition~\ref{lifting
condition} fails to be satisfied for all primes $\maxideal$ (see \cite[pp.\ 78-84]{berman}). Therefore, a correct formulation of \cite[Theorem~6.1]{duSL} requires Condition~\ref{lifting
condition} just as our Theorem~\ref{main theorem} does. 

Motivated by the examples presented in Section~\ref{counterexample section}, we ask the following.

\begin{question}\label{conjecture proisom zeta fn without fn eq}
Does there exist a finitely generated torsion-free nilpotent group $\Gamma$ such that, for almost all primes, the local zeta function $\zeta^{\wedge}_{\Gamma,p}(s)$ does not satisfy a functional equation?
\end{question}

To explain why this might be so, we briefly describe the relationship between zeta functions of algebraic groups and proisomorphic zeta functions of $\tgroup$-groups. Consider a
nilpotent Lie ring $L$ over $\integers$ of finite rank $n$. Put
$L_p=L\otimes\Zp$ and ${\curlL=L\otimes\Qp}$. If
$H$ is both a Lie subring and a full $\Zp$-sublattice of
$L_p$ (written ${H\leq L_p}$), we write $H\cong L_p$ if there exists
$g\in \Aut\curlL$ such that $L_p.g=H$. 
The following
two results connect zeta functions of groups with zeta functions of
algebraic groups:
\begin{proposition}[cf. \protect{\cite[Theorem~4.1]{GSS}}]\label{proposition proisomorphic group lie alg corr}
Given a $\tgroup$-group $\Gamma$, there exists a nilpotent Lie ring
$L$ of finite rank over $\integers$ such that for almost all primes
$p$,
\begindm
\zeta^{\wedge}_{\Gamma,p}(s)=\zeta^{\wedge}_{L,p}(s):=\sum_{{H\leq
L_p}\atop{H\cong L_p}}|L_p:H|^{-s}.
\enddm
\end{proposition}
\begin{proposition}[cf. \protect{\cite[Proposition~4.2]{GSS}}]\label{proposition proisomorphic lie alg alg gp corr}
Given a Lie ring $L$ of finite rank over $\integers$, let $\rho$ be
the representation $\Aut{\curlL}\to\GL{n}(\Qp)$ defined by fixing some
$\integers$-basis for $L$. Then $\Aut{\curlL}=\nonredG(\Qp)$, where $\nonredG$ is the algebraic automorphism group of $L\otimes \rationals$, and for each prime
$p$,
\begindm
\zeta^{\wedge}_{L,p}(s)=\intcurlGplus|\det{\rho(g)}|^s\mu_{\curlG}(g).
\enddm
\end{proposition}
Incidentally, Proposition~\ref{proposition proisomorphic lie alg alg gp corr} is of independent interest as it generalizes to counting isomorphic subrings in Lie rings $L\otimes \localringofintegers$, in which case the results of the present paper apply. A result of Bryant and Groves \cite[Theorem~A]{BryantandGroves}
implies that every algebraic group defined over $\rationals$
together with every possible faithful $\rationals$-rational
representation can be realized as the quotient of the automorphism group of some nilpotent Lie algebra by the group of $\mathrm{IA}$-automorphisms (these are the automorphisms acting trivially on the abelianisation of the Lie algebra), together with a natural representation. Unfortunately this does not give an immediate answer to Question~\ref{conjecture proisom zeta fn without fn eq}. Even in cases where the group of $\mathrm{IA}$-automorphisms coincides with the unipotent radical, we do not yet have sufficient understanding of the effect of the latter on the integral to extend our counterexamples to proisomorphic zeta functions of $\tgroup$-groups using the approach of \cite{BryantandGroves}.

Our results contribute to a broader picture of zeta functions of $\tgroup$-groups. Local functional equations are known to be satisfied for all subgroup counting zeta functions \cite{vollfneq} and for normal subgroup counting zeta functions of groups of class at most $2$ \cite{vollfneq, paajanenfneq}, while counterexamples are known in the normal subgroup case already in class $3$ \cite{duSW}. On the other hand, there are examples of $\tgroup$-groups whose local zeta functions of both subgroup and normal subgroup type are non-uniform \cite{duS-ec}. These examples corroborated an analysis of the zeta functions $\zeta^*_{\Gamma,p}(s)$ ($*\in\{\leq,\normal\}$) in \cite{duSG-annals} which established a theoretical link between counting subgroups of a $\tgroup$-group $\Gamma$ and counting $\Fp$-points on an algebraic variety associated to $\Gamma$. However, this analysis was not carried out for proisomorphic zeta functions, and it remains to be seen whether they too can exhibit non-uniform behaviour. In view of the dichotomy between $\zeta^\leq_{\Gamma,p}(s)$ and $\zeta^\normal_{\Gamma,p}(s)$, it is an interesting open question whether proisomorphic zeta functions satisfy local functional equations in general. If Question~\ref{conjecture proisom zeta fn without fn eq} is answered in the affirmative, it will be of great interest to characterize those $\tgroup$-groups whose local proisomorphic zeta functions do not satisfy functional equations. On the other hand, Theorem~\ref{main theorem} extends the known classes of algebraic groups and representations for which the local zeta functions are uniform and do satisfy functional equations. In view of Propositions~\ref{proposition proisomorphic group lie alg corr} and \ref{proposition proisomorphic lie alg alg gp corr}, this implies corresponding results for a larger class of $\tgroup$-groups than previously known (cf. \cite[Theorem B]{duSL}).

The paper is laid out as follows. In Section~\ref{reciprocity section} we consider generating functions over polyhedral cones similar to those considered by Stanley \cite{Stanley}. The chief difference is that in our context the summand varies within a finite number of cells into which the cone is subdivided. We combine a new combinatorial result with techniques of Stanley \cite{Stanley} and Igusa \cite{Igusa} to show that if the cone is simple then a certain sum of generating functions weighted by elements of an abstract Weyl group satisfies a reciprocity property. In Section~\ref{reduction section} we explain how the work of du Sautoy and Lubotzky \cite{duSL} can be extended to obtain more delicate control over the form of the reduced integral expression for $\ZcurlGrhops$ obtained by restricting the domain of integration to a connected reductive subgroup. In Section~\ref{bruhat section} we use a $\maxideal$-adic Bruhat decomposition due to Iwahori and Matsumoto to obtain a combinatorial expression for the zeta function as a sum of generating functions over a polyhedral cone, utilizing methods of \cite{Igusa} and \cite{duSL}. We are then able to deduce Theorem~\ref{main theorem} in Section~\ref{maintheorem section}, utilizing the results of Section~\ref{reciprocity section}. In Section~\ref{counterexample section} we give examples of local zeta functions of algebraic groups that do not satisfy functional equations.\\

\emph{Acknowledgements.} I wish to thank Marcus du Sautoy, my
doctoral supervisor, for many inspiring conversations. I am also
indebted to Nir Avni, Benjamin Klopsch, Uri Onn, Pirita Paajanen, Ilya Tyomkin and Christopher
Voll for valuable discussions and suggestions. In particular, I wish
to acknowledge Nir Avni for contributing to the construction of the
examples in Section~\ref{counterexample section} and to the proof of
Lemma~\ref{theta constant on double cosets lemma}. I am grateful to
my father, Peter Berman, and to Christopher Voll for their feedback on the presentation. 
During the course of this work I obtained support from several sources. I wish to thank the Rhodes Trust for a scholarship, Alex Lubotzky for a postdoctoral grant, and the Lady Davis Fellowship Trust for a Golda Meir postdoctoral fellowship. For financial support, I am grateful to Merton College and the Institute of Mathematics, University of Oxford, and to the Einstein Institute of Mathematics, The Hebrew University of Jerusalem. I also thank all of them for their hospitality. Finally, I wish to thank the referees for their insightful and helpful remarks.\\

We use the following notation:
\begindm
\begin{array}{l|l}
\naturals & \mbox{the set of positive integers}\\
\naturalszero & \mbox{the set of non-negative integers}\\
\Qp & \mbox{the set of $p$-adic numbers}\\
\Zp & \mbox{the set of $p$-adic integers}\\
\protect{[n]} & \mbox{the set}\ \ \{1,\ldots,n\}\\
\Sym(n)&\mbox{the symmetric group on \{1,\ldots,n\}} \\
|.|&\mbox{the $\maxideal$-adic absolute value} \\
v(x)\ & \mbox{the valuation of}\ x,\ \mbox{for}\ x\in \localfield \\
\end{array}
\enddm

\section{Reciprocities of generating functions over cones}\label{reciprocity section}

We begin by fixing our terminology, most of which is standard. The
dimension of a (non-empty) subset of $\reals^m$ is the dimension of
the subspace spanned by it. A (linear) hyperplane $H$ is a set of
the form $\{v\in\reals^m\suchthat v.w=0\}$, where $w$ is some fixed
vector in $\reals^m$. We will always assume such a vector has been
fixed even when this is not made explicit. We set $H^>:=\{v\in
\reals^m\suchthat v.w>0\}$ and similarly for $H^\geq$, $H^<$,
$H^\leq$, and we write $H^0=H$. We call a hyperplane rational if the
vector defining it has rational coordinates with respect to the
standard basis for $\reals^m$. We call a subset of $\reals^m$ a cone
if it is an intersection of closed half-spaces; that is, sets of the
form $H_i^\geq$ for some hyperplanes $H_i$. A cone is pointed if it
contains no lines and it is rational if it may be defined by rational hyperplanes. A cone is simplicial if it contains a finite subset
$S$ such that every point of the cone is uniquely expressible as a
non-negative $\reals$-linear combination of elements of $S$. It is simple if it is pointed and the set of lattice points contained in it is freely generated as a commutative monoid (in particular, a simple cone is simplicial). Finally, we
define a polyhedral cell complex. This is a cone $\curlC$ and a
family $\cellsofcomplex$ of cells defined by two finite collections
of hyperplanes -- bounding hyperplanes $\{B_i\}_{i\in M}$ and internal
hyperplanes $\{H_i\}_{i\in K}$. The cone of the complex is defined
to be 
\begindm
\curlC:=\bigcap_{i\in M}B_i^{\geq}.
\enddm
The cells are defined to be sets of the form
\begindm
\bigcap_{i\in M}B_i^{s_i}\cap\bigcap_{j\in K}H_j^{t_j}
\enddm
where each $s_i\in\{0,>\}$ and each $t_j\in\{0,<,>\}$. Thus every
cell is open in its support and is contained in $\curlC$; also, $\curlC$ is a disjoint union of the cells. For a subset $X$ of $\reals^m$, let $\overline{X}$ denote its closure with respect to the standard Euclidean metric. Given two
cells $F_1$, $F_2\in \cellsofcomplex$, we define $F_1$ to be a face
of $F_2$, written $F_1\leq F_2$, if $\overline{F_1}\subseteq
\overline{F_2}$. For each $I\subseteq M$
set
\begindm
{\displaystyle \curlC_I:=\bigcap_{i\in I}B^{>}_i\cap\bigcap_{j\in M\backslash
I}B^{\geq}_j.}
\enddm
For instance, $\curlC_{\emptyset}=\curlC$ and $\curlC_M=\mathrm{Int}(\curlC)$ (the interior taken with respect to
the support of $\curlC$). Set
\begindm
\cellsofcomplex_I:=\{F\in\cellsofcomplex\suchthat F\subseteq
\curlC_I\}.
\enddm
If $\bolde\in\curlC$, denote by $F_{\bolde}$ the unique cell in the
complex containing $\bolde$.

\begin{definition}\label{replaceable by constant definition}
We call a function $\loclinfnonZlattice:\cellsofcomplex\rightarrow
\integers^{m}$ \emph{piecewise constant} on the complex if
for each cell $F$ there exists $C_F\in \integers^{m}$ such that
$C_F.\bolde=\loclinfnonZlattice(F_{\bolde}).\bolde$ for all
$\bolde\in \overline{F}\cap\integers^{m}$.
\end{definition}

\begin{definition}\label{generating function definition}
For $Y\subseteq \curlC$, $\loclinfnonZlattice:\cellsofcomplex\rightarrow
\integers^{m}$, $A,B\in\integers^m$, $q$ a prime power and $s$ a complex variable put
\begin{eqnarray*}
\genfn_{Y,A,B,\gamma}(q,q^{-s})&:=&\sum_{\bolde\in
Y\cap\integers^{m}}q^{(A+\loclinpiece{\bolde}).\bolde-(B.\bolde)s}.
\end{eqnarray*}
\end{definition}
Usually we will simply write this as $\genfn_{Y}(q,q^{-s})$. We can now state our main combinatorial result:
\begin{theorem}\label{reciprocity theorem}
Let $(\curlC, \cellsofcomplex)$ be a polyhedral cell complex in
$\reals^m$ defined by rational hyperplanes, and suppose $\curlC$ is
a simplicial cone defined by $m$ bounding hyperplanes with $\dim\curlC=m$. Let
$\loclinfnonZlattice:\cellsofcomplex\to\integers^m$ be a piecewise constant function on the complex and let $A,B\in\integers^m$.
Suppose that $B.\bolde>0$ for all $0\neq\bolde\in \curlC$. Then for each $I\subseteq [m]$,
$\genfn_{\curlC_I}(q,q^{-s})$ is a rational function in $q, q^{-s}$. Furthermore,
\begin{eqnarray}\label{equation reciprocity of generating fn}
\genfn_{\curlC_I}(q,q^{-s})|_{q\to
q^{-1}}&=&(-1)^{m}\genfn_{\curlC_{[m]\backslash I}}(q,q^{-s}).
\end{eqnarray}
\end{theorem}

Note that Theorem~\ref{reciprocity theorem} reduces (essentially) to a special case of \cite[Theorem~4.6.14]{Stanley} if $I=\emptyset$ and there are no internal hyperplanes. On the other hand, if $I\notin \{\emptyset, [m]\}$, the reciprocity result is best possible in the sense that if $\curlC$ is not simplicial, (\ref{equation reciprocity of generating fn}) will not necessarily hold (it is easy to construct examples).

The proof of the theorem depends on the following two results. The first of these is new, as far as we are aware, while the second is essentially a restatement of \cite[Theorem~4.6.14]{Stanley} in a form suitable to our context.

\begin{proposition}\label{dimension counting proposition}
Take $(\curlC, \curlF)$ as in the hypothesis of Theorem~\ref{reciprocity theorem}. Then for each
$F_0\in \curlF$ and $I\subseteq [m]$ we have
\begindm
\sum_{F\in\curlF_I\atop F\geq F_0}(-1)^{\dim F}=\left\{
\begin{array}{ll}
0&\mbox{if}\ \ F_0\notin \curlF_{[m]\backslash I}\\
(-1)^{m}&\mbox{if}\ \ F_0\in \curlF_{[m]\backslash I}.
\end{array}
\right.
\enddm
\end{proposition}
\begin{proof}
We prove this by induction on the number $|K|$ of internal
hyperplanes. We start with the case $|K|=0$. If $F_0\in
\curlF_{[m]\backslash I}$ then $\{F\in \curlF_I\suchthat F_0\leq
F\}=\curlF_{[m]}$, which consists of a single cell of dimension $m$,
so we obtain the required expression. If $F_0\notin
\curlF_{[m]\backslash I}$, put $I_0=\max\{J\subseteq [m]\backslash
I\suchthat F_0\subseteq \curlC_{J}\}$ and note that
$\emptyset\subseteq I_0\subsetneq [m]\backslash I$. Then, since
$\curlC$ is simplicial, we have
\begin{eqnarray*}
\sum_{F\in\curlF_I\atop F\geq F_0}(-1)^{\dim F}&=&\sum_{J\subseteq
[m]\backslash (I\cup I_0)}(-1)^{m-|J|}\\
&=&0.
\end{eqnarray*}
Suppose now that $(\curlC, \curlF)$ is a cell complex with
internal hyperplanes $\{H_i\}_{i\in K}$. Let $H$ be a
hyperplane and let $(\curlC,\curlFext)$ be the cell complex formed
by $\curlC$ with internal hyperplanes $\{H_i\}_{i\in K}\cup \{H\}$.
For each $J\subseteq [m]$ define $\curlF_J$ (respectively,
$\curlFext_J$) as before. Inductively, we assume that the result
holds for $(\curlC, \curlF)$. Note that $\curlFext$ consists of all
non-empty cells of the form $F\cap H^*$ where $F\in \curlF$ and
$*\in\{0,<,>\}$. Furthermore, for each $*\in\{0,<,>\}$, $F\cap
H^*\in \curlF^2_{[m]\backslash I}$ if and only if $F\in
\curlF_{[m]\backslash I}$ and $F\cap H^*\neq\emptyset$. Thus it is
sufficient to show that if $F_0\in \curlF$, $*\in\{0,<,>\}$ and
$F_0\cap H^*\neq\emptyset$, the following holds:
\begin{equation}\label{invariance of sum of signs equation 1}
\sum_{{F\in \curlF_I}\atop{F_0\leq F}}(-1)^{\dim{F}}=\sum_{{F'\in
\curlFext_I}\atop{F_0\cap H^*\leq F'}}(-1)^{\dim{F'}}.
\end{equation}
If $F_0\cap H^<\neq\emptyset$ then for each $F\geq F_0$ in
$\curlF_I$ we have $F\cap H^<\neq\emptyset$ and hence $\dim F = \dim
F\cap H^<$. Thus
\begindm
\{F'\in \curlFext_I\suchthat F'\geq F_0\cap H^<\}=\{F\cap
H^<\suchthat F\geq F_0,\ F\in \curlF_I\}
\enddm
and (\ref{invariance of sum of signs equation 1}) follows in this
case. The case $F_0\cap H^>\neq\emptyset$ is similar. Suppose
that $F_0\cap H\neq\emptyset$. If $F'\geq F_0\cap H$ in
$\curlFext_I$ then $F'=F\cap H^*$ for some $F\geq F_0$ in $\curlF_I$
and $*\in\{0,<,>\}$.
In order to establish (\ref{invariance of sum
of signs equation 1}) it is sufficient to show that
\begin{equation}\label{invariance of sum of signs equation 2}
\sum_{{*\in\{0,<,>\}}\atop{F\cap H^*\neq\emptyset}} (-1)^{\dim{F\cap
H^*}}=(-1)^{\dim{F}}.
\end{equation}
Suppose that there exists $*\in\{0,<,>\}$ such that $F\subseteq
H^*$. Then (\ref{invariance of sum of signs equation 2}) follows
immediately. The alternative possibility is that $F\cap
H\neq\emptyset$ and $F\nsubseteq H$. In this case $F\cap
H^>\neq\emptyset$ and $F\cap H^<\neq\emptyset$. We have $\dim F\cap
H^>=\dim F\cap H^< = \dim(F\cap H)+1=\dim F$, and (\ref{invariance
of sum of signs equation 2}) follows.
\end{proof}

\begin{lemma}\label{inversion of a cell lemma}
Take $(\curlC,\curlF)$, $A$, $B$ and $\gamma$ as in the hypothesis of Theorem~\ref{reciprocity theorem}. Then
for each $F_0\in \curlF$, $\genfn_{F_0}(q,q^{-s})$ is a rational function in $q$ and
$q^{-s}$. Furthermore,
\begin{eqnarray}\label{reciprocity for a cell equation}
\genfn_{F_0}(q,q^{-s})|_{\pgoestopinverse}&=&(-1)^{\dim
F_0}\sum_{F\leq F_0}\genfn_{F}(q,q^{-s}).
\end{eqnarray}
\end{lemma}
\begin{proof}
Since $\gamma$ is piecewise constant, there exists
$C_{F_0}\in\integers^{m}$ such that ${C_{F_0}.\bolde=\loclinfnonZlattice(\cellassignfn_{\bolde}).\bolde}$
for all $\bolde\in \overline{F_0}\cap\integers^{m}$. We begin by
proving the result in the case that $\overline{F_0}$ is simplicial. In this
case, there exist $u_1,\ldots,u_k\in\integers^{m}$ (where
$k=\dim{F_0}$) such that each element of $F_0$ is uniquely
expressible as a positive real linear combination of these $k$
vectors. Denote by $\funddom{1}$ the set $\integers^{m}\cap
\sum_{i=1}^k (0,1]u_i$ and by $\funddom{0}$ the set
$\integers^{m}\cap \sum_{i=1}^k [0,1)u_i$. Then
\begindm
\integers^{m}\cap
F_0=\coprod_{{\bolde\in\funddom{1}}\atop{\boldt\in\naturalszero^{k}}}\left(\bolde+\sum_{i=1}^k
t_i u_i\right)
\enddm
so
\begindisplayarray{llll}
\genfn_{F_0}(q,q^{-s})&=&\sum_{\bolde\in \integers^{m}\cap F_0}q^{(A+\loclinpiece{0}).\bolde -(B.\bolde) s}&\\
&=&\left(\sum_{{\bolde}\in\funddom{1}}q^{(A+C_{F_0}).{\bolde} -(B.{\bolde})s}\right)\prod_{i=1}^k \frac{1}{1-q^{(A+C_{F_0}).u_i-(B.u_i)s}}.&\\
\enddisplayarray
Note that for each $u_i$, $i=1,\ldots,k$ we have $B.u_i>0$, so the sum converges for suitable $s$. Thus
\begin{eqnarray*}
\genfn_{F_0}(q,q^{-s})|_{\pgoestopinverse}&=&(-1)^{\dim{F_0}}\left(\sum_{{\bolde}\in\funddom{0}}q^{(A+C_{F_0}).{\bolde} -(B.{\bolde})s}\right)\prod_{i=1}^k \frac{1}{1-q^{(A+C_{F_0}).u_i -(B.u_i)s}}\\
&=&(-1)^{\dim F_0}\sum_{\bolde\in \integers^{m}\cap
\overline{F_0}}q^{(A+\loclinfnonZlattice(F_{\bolde})).\bolde -(B.\bolde)s}\\
&=&(-1)^{\dim F_0}\sum_{F\leq F_0}\genfn_{F}(q,q^{-s}).
\end{eqnarray*}

We now consider the general case. Since $\curlC$ is a pointed cone, so is $\overline{F_0}$. By \cite[Lemma 4.6.1]{Stanley},
there exists a triangulation of $\overline{F_0}$, namely an
expression for $\overline{F_0}$ as a union of simplicial cones closed
under intersection and under taking faces. Comparing (\ref{reciprocity for a cell equation}) (for $F_0$ simplicial) to
Stanley's Lemma~4.6.13, and noting that
$\overline{F_0}=\disjointunion_{F\leq F_0}F$, our result now follows
from the proof of Stanley's Theorem~4.6.14.
\end{proof}

We have
\begin{align*}
\genfn_{\curlC_I}(q,q^{-s})|_{q\rightarrow q^{-1}}&=\sum_{F\in
\curlF_I}\genfn_{F}(q,q^{-s})|_{\pgoestopinverse}\\
&=\sum_{F'\in
\curlF_I}\sum_{F\leq F'}(-1)^{\dim F'}\genfn_{F}(q,q^{-s})&&\text{(by Lemma~\ref{inversion of a cell lemma})}\\
&=\sum_{F\in\curlF}\genfn_{F}(q,q^{-s})\sum_{{F'\in
\curlF_I}\atop{F\leq F'}}(-1)^{\dim
F'}\\
&=\sum_{F\in\curlF_{[m]\backslash I}}(-1)^{m}\genfn_{F}(q,q^{-s})&&\text{(by Proposition~\ref{dimension
counting proposition})}\\
&=(-1)^{m}\genfn_{\curlC_{{[m]}\backslash I}}(q,q^{-s}).
\end{align*}
This completes the proof of Theorem~\ref{reciprocity theorem}.

Now let $W$ denote the Weyl group of an abstract root system $\Phi$ and let $\alpha_1,\ldots \alpha_l$ be a fundamental system of roots.
Let $\Phi^+, \Phi^-$ denote the sets of positive, negative roots respectively. We assume that $l<m$. In Section~\ref{functional equation subsection}, where we derive our functional equations for the local zeta functions of $\nonredG$, we will need to consider not just a single generating function as in Theorem~\ref{reciprocity theorem}, but a weighted sum over the Weyl group associated to $\nonredG$ of such generating functions. The cone $\curlC$ will defined by $m=l+d$ bounding hyperplanes, $l$ of these corresponding to the fundamental roots, and $d$ corresponding to the rank of the maximal central torus of $\nonredG$. In our weighted sum each summand will be a generating function over a subset of $\curlC$ with certain bounding hyperplanes (corresponding to a subset of the set of fundamental roots) excluded.

For each $w\in W$ set
\begin{eqnarray}\label{equation definition of I_w}
I_w&:=&\{i\in [l]\suchthat\alpha_i\in w(\Phi^{-})\}.
\end{eqnarray}
The weighted sum is as follows.
\begin{definition}
\begin{eqnarray*}
Z(q,q^{-s})&:=&\sum_{w\in
W}q^{-\lambda(w)}\genfn_{\curlC_{I_w}}(q,q^{-s}).
\end{eqnarray*}
\end{definition}
We now adapt the proof of \cite[Theorem~5.9]{duSL} to show that, under specified conditions, Theorem~\ref{reciprocity theorem} implies that $Z(q,q^{-s})$ satisfies a functional equation. This reflects an interplay between a reciprocity of the generating function (Theorem~\ref{reciprocity theorem}) and a symmetry of the Weyl group. This interplay lies at the heart of the functional equation.
\begin{corollary}\label{fn eq corollary}
Suppose there exists a vector $\boldanought\in
\integers^{m}\cap\curlC\cap\bigcap_{i\in K} H_i$ such that for each
$I\subseteq [l]$ we have
\begin{eqnarray}\label{cone tranlation equation}
\curlC_{[m]\backslash
I}\cap\integers^{m}&=&\boldanought+\curlC_{[l]\backslash
I}\cap\integers^{m}.
\end{eqnarray}
Then, under the hypotheses of Theorem~\ref{reciprocity theorem},
\begin{eqnarray*}
Z(q,q^{-s})|_{\pgoestopinverse}&=&(-1)^{m}q^{|\Phi^+|\,+(A+\loclinfnonZlattice(\cellassignfn_{\boldanought})).\boldanought-(B.\boldanought)s}Z(q,q^{-s}).
\end{eqnarray*}
\end{corollary}
\begin{proof}
\begin{align*}
Z(q,q^{-s})|_{\pgoestopinverse}&=(-1)^{m}\sum_{w\in
W}q^{\lambda(w)}\EcurlCpps{[m]\backslash I_w}\ \ \ \ \ \ \ \ \ \ \ \ \ \ \  \text{(by Theorem~\ref{reciprocity theorem})}\\
&=(-1)^{m}q^{(A+\loclinfnonZlattice(\cellassignfn_{\boldanought})).\boldanought-(B.\boldanought)s}\sum_{w\in
W}q^{\lambda(w)}\EcurlCpps{[l]\backslash I_w}\\
&=(-1)^{m}q^{(A+\loclinfnonZlattice(\cellassignfn_{\boldanought})).\boldanought-(B.\boldanought)s}\sum_{w\in
W}q^{|\Phi^+|-\lambda(ww_0)}\EcurlCpps{I_{ww_0}}\\
&=(-1)^{m}q^{|\Phi^+|\,+(A+\loclinfnonZlattice(\cellassignfn_{\boldanought})).\boldanought-(B.\boldanought)s}Z(q,q^{-s}),
\end{align*}
where $w_0\in W$ is the longest element of the Weyl group. To see the second equality, note that by (\ref{cone tranlation equation}) it is enough to check that for every cell $F$ there exists a cell $F'$ such that $F\leq F'$ and $\boldanought\in\overline{F'}$ and then to observe that $\loclinfnonZlattice$ may be replaced by a constant function on $\overline{F'}$. This follows from the fact that $\curlC$ is a union of the chambers in $\cellsofcomplex$, each of which contains $\curlC\cap\bigcap_{i\in K}H_i$ (we define a chamber to be the closure of a cell which is maximal with respect to the face relation).
The third equality follows from the standard properties
$\lambda(w)+\lambda(ww_0)=|\Phi^+|$
and $w_0(\Phi^-)=\Phi^+$ (see \cite[Section~1.8]{humphreys2}). The latter gives
\begin{eqnarray*}
[l]\backslash I_w&=&\{i\in [l]\suchthat\alpha_i\notin w(\Phi^-)\}\\
&=&\{i\in [l]\suchthat\alpha_i\in ww_0(\Phi^-)\}\\
&=&I_{ww_0}.
\end{eqnarray*}
\end{proof}

\section{Integrals over reductive groups}\label{reduction section}
Let $\nonredG_0$ denote the connected component of $\nonredG$. By
\cite[Lemma~4.1 and Proposition~2.1]{duSL} we have that for almost
all $\maxideal$, $Z_{\nonredG,\rho,{\maxideal}}(s)=Z_{\nonredG_0,\rho,{\maxideal}}(s)$.
We therefore assume throughout that $\nonredG$ is connected. Fix a
reductive subgroup $G$ of $\nonredG$ such that $\nonredG=N\rtimes
G$, where $N$ is the unipotent radical of $\nonredG$. Fix a maximal torus $T$ of $G$. We impose the following restriction.

\begin{condition}\label{splitting condition}
The maximal torus $T$ of $G$ splits over $\globalfield$; that is,
there exists a $\globalfield$-isomorphism $\phi:T\to
\Gmult^{\dim{T}}$.
\end{condition}

This implies in particular that $T$ splits over $\localfield$ for every prime $\maxideal$.
One of the key observations made by du Sautoy and Lubotzky is that
one is free to choose an equivalent $\globalfield$-rational
representation. For a particular localization $\globalfield_{\maxideal}$, this may change the
integral; however, they showed that for almost all $\maxideal$ it
will not. Specifically, they proved the following.

\begin{lemma}[\protect{\cite[Lemma 4.2]{duSL}}]\label{equivalent representation
lemma} Let $\rho': \nonredG\to \GLn$ be a $\globalfield$-rational
representation equivalent to $\rho$; that is, there exists $A\in
\GLn(\globalfield)$ such that $\rho'(x)=A\rho(x)A^{-1}$ for all
$x\in \nonredG(\globalfield)$. Then, for almost all primes $\maxideal$,
\begindm
Z_{\nonredG,\rho',\maxideal}(s)=Z_{\nonredG,\rho,{\maxideal}}(s).
\enddm
\end{lemma}

For our purposes we will require the representation $\rho$ to
satisfy number of properties. In view of Lemma~\ref{equivalent
representation lemma} it will be sufficient to show that these
properties are satisfied for some equivalent $\globalfield$-rational
representation. This will ensure that we can pass to this equivalent
representation for almost all primes $\maxideal$ without changing the
integral.

\begin{definition}\label{good representation definition}
We call a $\globalfield$-rational representation $\rho$ `good' if it
satisfies the following. Let the underlying ordered basis for the
representation space $V$ be $(u_i)_{i\in [n]}$. There exists a
decomposition of $[n]$ as $[n]=\disjointunion_{i=1}^c I_i$ giving a
decomposition of $V=\bigoplus_{i=1}^c U_i$, where $U_i$ is the
subspace of $V$ spanned by $(u_j)_{j\in I_i}$, so that, putting
$V_i=\bigoplus_{j=i}^{c} U_j$, we have
\begin{enumerate}
\item each $V_i$ is $\nonredG$-stable
\item each $U_i$ is $G$-stable
\item $N$ acts trivially on each section $V_i/V_{i+1}$
\item $T$ acts diagonally with respect to the basis $(u_i)_{i\in [n]}$
\item each $U_i$ is an irreducible subrepresentation of $\rho|_G$.
\end{enumerate}
\end{definition}
\begin{proposition}
Given a faithful
$\globalfield$-rational representation $\rho$ of $\nonredG$ there exists an
equivalent faithful $\globalfield$-rational representation $\rho'$
which is good and satisfies, for almost all primes,
\begindm
Z_{\nonredG,\rho',{\maxideal}}(s)=Z_{\nonredG,\rho,{\maxideal}}(s).
\enddm
\end{proposition}
\begin{proof}
By \cite[Lemma~4.3]{duSL}, there exists a representation equivalent
to $\rho$ whose underlying ordered basis gives a decomposition
$V=\bigoplus_{i=1}^c U'_i$ satisfying properties (1), (2) and (3)
above. Each $G$-stable subspace $U'_i$ can be decomposed further
into irreducible components under the $G$-action. This gives a
refinement of the original decomposition, say $V=\bigoplus_{i=1}^{c}
U_i$, which is easily seen to satisfy properties (1), (2), (3) and
(5). It remains to note that since $T$ splits over $\globalfield$,
there exists for each $i$ a basis $(u_j)_{j\in I_i}$ for $U_i$ on
which $T$ acts diagonally. The result now follows from
Lemma~\ref{equivalent representation lemma}.
\end{proof}
We now assume that $\rho$ is good and has the form described in
Definition~\ref{good representation definition}. We recall the
following definitions from \cite{duSL}. For each $i$ let
$N_i$ denote the kernel of the action of $N$ on $V/V_{i+1}$ (so we obtain a normal series with $N_1=N$ and $N_c=1$). By
identifying $\curlU_i:=U_1\oplus\ldots\oplus U_i$ with $V/V_{i+1}$
we obtain a faithful representation $\psi_i: N/N_i\to
\mathrm{GL}(\curlU_i)$. This defines a unique representation
$\varphi_i:\nonredG/N_i\to\GL{n}$ satisfying
\begindm
\begin{array}{llll}
\varphi_i(nN_i)(v)&=&\psi_i(nN_i)(v)& \ \mbox{for all}\ \ n\in N, v\in\curlU_i\\
\varphi_i(nN_i)(v)&=&v& \ \mbox{for all} \ \ n\in N, v\in {V_{i+1}}\\
\varphi_i(gN_i)(v)&=&\rho(g)(v)& \ \mbox{for all} \ \ g\in G, v\in V.
\end{array}
\enddm

If $X/N_i\leq \nonredG/N_i$ put $(X/N_i)^+:=\varphi_i^{-1}[\varphi_i((X/N_i)(\localfield))\cap(\Mn(\localringofintegers))]$. As in \cite{duSL}, we assume that the following
condition holds:

\begin{condition}\label{lifting condition}
For each $i\in [c]$ and for each $\overline{g}\in (\nonredG/N_i)^+$
there exists $g\in \nonredG^+$ such that $gN_i=\overline{g}$.
\end{condition}

This appears as Assumption~2.3 in \cite{duSL}. In the same paper, the authors state that the condition holds for almost all primes \cite[Corollary~4.5]{duSL}. As pointed out in Section~\ref{introduction section}, this is incorrect. We now explain how the integral may be reduced under Condition~\ref{lifting condition}. Let $\mu_N$ denote the right Haar
measure on $N$ normalized such that
$\mu_N(N(\localringofintegers))=1$. Each section $N_i/N_{i+1}$
admits a right Haar measure $\mu_{N_i/N_{i+1}}$ satisfying
$\mu_{N_i/N_{i+1}}(({N_i/N_{i+1}})^{+})=1$. This can be used
to define certain functions which describe an action of the
reductive subgroup $G$ on each section $N_{i}/N_{i+1}$:
\begin{definition}
For each $h\in G(\localfield)$ and $i\in [c-1]$ put
\begindm
\theta_i(h):=\mu_{N_i/N_{i+1}}(\{n\in
N_i/N_{i+1}\suchthat nh\in (\nonredG/N_{i+1})^+ \}).
\enddm
\end{definition}
Recall from (\ref{definition first of theta}) that $\theta(h):=\mu_N(\{n\in N\suchthat nh\in\nonredG^+\})$. In \cite{duSL}, du Sautoy and Lubotzky reduced the integral
$\ZcurlGrhops$ to an integral over the reductive subgroup~$G$. It
will be convenient for us to divide their result into two parts.

\begin{proposition}[cf. (2.1) in the proof of \protect{\cite[Theorem~2.2]{duSL}}] \label{proposition product of thetas} Suppose that  $(\nonredG,\rho)$ satisfies
properties (1), (2) and (3) in Definition~\ref{good representation
definition} and Condition~\ref{lifting condition} with respect to
some suitable decomposition of the representation space. Then
\begindm
\theta(h)=\prod_{i=1}^{c-1}\theta_i(h).
\enddm
\end{proposition}

\begin{definition}\label{definition reduced integral}
Put
\begindm
\ZGrhops:=\int_{G^+}|\detrho(h)|^s\theta(h)\mu_{G}(h).
\enddm
\end{definition}

\begin{theorem}[\protect{cf. \cite[Proof of Theorem~2.2]{duSL}}] \label{theorem reduction of integral}
We have
\begindm
\ZnonredGrhops=\ZGrhops.
\enddm
\end{theorem}

Note that in contrast to Proposition~\ref{proposition product of thetas}, Theorem~\ref{theorem reduction of integral} is unconditional. In Section~\ref{bruhat section} when we come to analyze the integral using a $\maxideal$-adic Bruhat decomposition, we will need to understand how $\theta$ varies on Bruhat double cosets. Also, it will be crucial to express $\theta|_{T}$ in terms of characters of $T$. We now prove two results about the action of $G$ on $N$ which will enable us to deal with these issues.\\

\begin{lemma}\label{theta constant on double cosets lemma}
For all $h\in G(\localfield)$ and $h_1, h_2\in G(\localringofintegers)$,
$\theta(h)=\theta(h_1 h h_2)$.
\end{lemma}
\begin{proof}
For $g\in G(\localfield)$ put $M_g=\{n\in N\suchthat ng\in \nonredG^+\}$.
It follows that $M_{h_1 h h_2}^{h_1}=M_h$. Consider a
measure $\mu'_N$ on $N$ given by $\mu'_N(A)=\mu_N(A^{h_1})$
for all measurable sets~$A$. This defines
a right Haar measure on $N$. By uniqueness of the Haar measure there
exists $\lambda\in\reals_{>0}$ such that $\mu'_N=\lambda\mu_N$.
However,
$\mu'_N(N(\localringofintegers))=\mu_N(N(\localringofintegers)^{h_1})=\mu_N(N(\localringofintegers))=$1, so $\lambda=1$. It follows that $\theta(h)=\mu_N(M_h)=\mu_N(M_{h_1 h
h_2})=\theta(h_1 h h_2)$.
\end{proof}
Next we analyze $\theta|_T$ in the case that $\rho$ is good (cf. Definition~\ref{good representation definition}) and $(\nonredG,\rho)$ satisfies Condition~\ref{lifting condition}. There
exists a natural identification of $N_i/N_{i+1}$ with a subspace of
$U_{i+1}^{s_i}$, where $s_i=\sum_{j=1}^{i}\dim{U_j}$. This comes
from considering the action of $N_i/N_{i+1}$ on $V/V_{i+2}$. Since
the induced action on $V/V_{i+1}$ is trivial, $N_i/N_{i+1}$ is
identified with an additive algebraic subgroup of
$U_{i+1}^{s_i}\cong\Gadd^{s_i(s_{i+1}-s_i)}$. Note that $n\in
(N_i/N_{i+1})^+$ if and only if its image in $U_{i+1}^{s_i}$ is in
the $\localringofintegers$-span of the basis $(\{u_j\}_{j\in
I_{i+1}})^{s_i}$. Now $G$ acts on $U_{i+1}$, hence on $U_{i+1}^{s_i}$;
in fact if $n\in N_i/N_{i+1}$, $g\in G$ and $v$ is the image of
$n$ in $U_{i+1}^{s_i}$, then $ng\in\left(\nonredG/N_{i+1}\right)^+$
if and only if $v.g$ is contained in the $\localringofintegers$-span of $(\{u_j\}_{j\in
I_{i+1}})^{s_i}$. Thus the question of integrality is reduced to a
question in linear algebra.

\begin{definition}\label{definition T sigma}
For each $\sigma\in\Sym(n)$, put
\begindm
T_{\sigma}(\localfield):=\{t\in
T(\localfield)\suchthat
\valuation{\lambda_{\sigma(i)}(t)}\leq\valuation{\lambda_{\sigma(j)}(t)}\
\ \mbox{for all}\ \ 1\leq i < j\leq n\},
\enddm
where $\lambda_i(t)$ is the eigenvalue for the action of $t$ on the
basis element $u_i$.
\end{definition}

\begin{lemma}\label{local linearity of theta lemma}
Let $\rho$ be a good representation of $\nonredG$ and suppose that,
for almost all primes $\maxideal$, $(\nonredG,\rho,{\maxideal})$ satisfies
Condition~\ref{lifting condition}. Define $\theta$ with respect to
$\rho$ as described above. Then for each $\sigma\in\Sym(n)$, there
exist non-negative integers $m_i(\sigma)$ $\nobreak{(i=1,\ldots,n)}$ such
that, for almost all $\maxideal$, for all $t\in T_{\sigma}(\localringofintegers)$,
\begindm
\theta(t)=|\lambda_1(t)|^{-m_1(\sigma)}\ldots
|\lambda_n(t)|^{-m_n(\sigma)}.
\enddm
\end{lemma}
\begin{proof}
We fix a one-to-one correspondence
$\tau\mapsto X_{\tau}$ between elements of \\ $\Sym(s_i(s_{i+1}-s_i))$ and
orderings of the basis $X=(\{u_j\}_{j\in I_{i+1}})^{s_i}$ for
$U_{i+1}^{s_i}(\globalfield)$. Fix an ordered basis $(v_j)_{j\in [m]}$ for $N_i/N_{i+1}(\globalringofintegers)$, where $m:=\dim{N_i/N_{i+1}}$. Given $\nobreak{\tau\in\Sym(s_i(s_{i+1}-s_i))}$, there exists
an ordered basis $(w_j(\tau))_{j\in [m]}$ for $N_i/N_{i+1}(\globalfield)$ with the following property:\\

(*) For all $j\in[m-1]$ and $k\in [s_i(s_{i+1}-s_i)-1]$, if $w_{j}(\tau)$ has zero projection on the subspace generated by the first $k$ basis elements of $X_{\tau}$, then $w_{j+1}(\tau)$ has zero projection on the subspace generated by the first $k+1$ basis elements of $X_{\tau}$.\\

For each $\tau$ we fix such an ordered
basis for $N_i/N_{i+1}(\globalfield)$ and denote it by $Y_{\tau}$. Let $\Delta_i(\tau)$ denote the linear map $N_i/N_{i+1}(\globalfield)\to N_i/N_{i+1}(\globalfield)$ given by $v_j\mapsto w_j(\tau)$ for $j=1,\ldots,m$. Note that if $t\in T$, $t$ acts diagonally on
the basis $X$ for $U_{i+1}^{s_i}(\globalfield)$, since $\rho$ is good.
Given $\sigma$, there exists ${\tau}\in \Sym(s_i(s_{i+1}-s_i))$ such that,
for all $t\in T_{\sigma}(\localringofintegers)$, the valuations of
the eigenvalues for the action of $t$ on the ordered basis
$X_{\tau}$ are in non-decreasing order. Fix some such suitable
${\tau}$ (there is some freedom here which need not concern us) and write the eigenvalues as
$\nu_1(t),\ldots,\nu_{s_i(s_{i+1}-s_i)}(t)$.
Now fix $t\in T_{\sigma}(\localringofintegers)$. Let $Y_{\tau}$ be
the basis for $N_i/N_{i+1}(\globalfield)$ chosen above and
note that, for almost all $\maxideal$, the $\localringofintegers$-span of $Y_{\tau}$ is
precisely $N_i/N_{i+1}(\localringofintegers)$. This follows from the fact that the transformation $\Delta_i(\tau)$ defined above lies (for almost all $\maxideal$) in $\GL{m}(\localringofintegers)$. By construction,
$Y_{\tau}$ satisfies (*) with respect to $X_{\tau}$.
We may further assume that, in the expression for each element of
$Y_{\tau}$ as a linear combination of the elements of the basis
$X_{\tau}$, every coefficient has valuation zero (this holds for
almost all $\maxideal$). It now follows immediately that
$\theta_i(t)$ has the form
$\prod_{j=1}^{s_i(s_{i+1}-s_i)}|\nu_j(t)|^{-\delta_j}$ where each
$\delta_j\in\{0,1\}$. Note that for each $j\in [s_i(s_{i+1}-s_i)]$ there
exists $k\in [s_{i+1}]\backslash [s_i]$ such that $\nu_j(t)=\lambda_k(t)$, where
$\lambda_k(t)$ is the eigenvalue for the action of $t$ on $u_k$, so by Proposition~\ref{proposition product of thetas}
we obtain an expression of the required form by taking the product
of the expressions obtained for each $\theta_i(t)$. It remains to
note that this construction provides non-negative integers
$m_j(\sigma)$ depending only on $\sigma$; in particular they are
independent of $t\in T_{\sigma}(\localringofintegers)$ and of the
localization $\localfield$.
\end{proof}

\section{A combinatorial expression for $Z_{G,\rho,{\maxideal}}(s)$}\label{bruhat section}

We begin by recalling the set-up of \cite{duSL} in the reductive
case (with some important distinctions). Let $\Phi$ be the root system of $G$ relative to $T$ (that is, those elements of $\HomTGm$ which give rise to non-trivial weights for the adjoint action on the Lie algebra of $G$).  As in Section~\ref{reciprocity section}, let $\alpha_1,\ldots,\alpha_l$ be a set of fundamental roots for $\Phi$ and let $W$ be the Weyl group. Let $S$ be the maximal central torus of $G$ and put $d:=\dim{S}$. It is well-known that $G=S.G'$, where $G'$ is the derived subgroup of $G$. Then $G'$ is semisimple (see, for instance, \cite[p.\ 10]{Borel}). The root systems of $G$ and $G'$ are isomorphic, hence by \cite[26.2, Corollary B (f)]{humphreys1} the semisimple rank of $G$ is $l$. It follows that the rank of $G$ is $l+d$. To each root $\alpha$ there corresponds a
minimal closed unipotent subgroup $U_\alpha$ of $G$ such that conjugation by elements of $T$ maps $U_\alpha$ into itself and there exists an isomorphism $\theta_\alpha:\Gadd\to U_\alpha$ satisfying $t \theta_\alpha (x) t^{-1} =\theta _\alpha (\alpha (t) x)$ for all $x\in \Gadd$, $t\in T$ (see \cite[26.3, Theorem]{humphreys1}). By Condition~\ref{splitting condition} we have a $\globalfield$-isomorphism $\phi: T\to \Gmult^{l+d}$. We will need the following. 

\begin{condition}\label{good reduction condition}
The group $G$, the maximal torus $T$, and the isomorphisms
${\phi:{T\to(\Gmult)^{l+d}}}$ and $\theta_\alpha:\Gadd\to U_\alpha$
($\alpha\in \Phi$) all have good reduction mod $\unifparam$.
\end{condition}

This was assumed also in \cite{Igusa}. As observed in \cite[p.\ ~74]{duSL}, it holds for almost all primes $\maxideal$. By good reduction of the maps we mean that reducing mod $\pi$ gives induced maps ${T(\localringofintegers/\maxideal)\to (\localringofintegers/\maxideal)^{l+d}}$ and $\localringofintegers/\maxideal\to U_\alpha(\localringofintegers/\maxideal)$ which are isomorphisms.
The finite Weyl group $W$ of $G$ is isomorphic to $N(T)/T$, where
$N(T)$ is the normalizer of $T$ in $G$. We define actions of $W$ on
$\HomTGm$ and on $\HomGmT$ as follows: if $w\in W$, $\alpha\in
\HomTGm$ and $\xi\in\HomGmT$, put $(w\alpha)(t)=\alpha(w^{-1}(t))$ for
all $t\in T$ and put $(w\xi)(\tau)=w(\xi(\tau))$ for all
$\tau\in\Gmult$, where the action of $W$ on $T$ is by conjugation.
In particular, the former induces an action of $W$ on $\Phi$. A
consequence of Condition~\ref{good reduction condition} is that it
is possible to take a coset representative $g_w$ for each $w\in
W=N(T)/T$ such that $g_w\in N(T)(\localringofintegers)$ (see
\cite[p.\ 697]{Igusa}). We now put $\Xi:=\HomGmT$ and write
$\curlV:=\HomTGm\tensor_{\integers} \reals$,
$\curlV^*:=\HomGmT\tensor_{\integers} \reals$. Note that $\Xi\cong \integers^{l+d}\cong \HomTGm$ since $G$ has rank $l+d$.

For each $\alpha\in\HomTGm$ and $\xi\in\HomGmT$ let $\alphapairxi$
be the integer satisfying $\alpha(\xi(\tau))=\tau^{\alphapairxi}$
for all $\tau\in\Gmult$. This provides a pairing
$\HomTGm\times\HomGmT\to\integers$ given by
$(\alpha,\xi)\mapsto\alphapairxi$ which extends uniquely to a linear
map $\curlV\times \curlV^*\to\reals$. There
exists a finite set $\Phi^*$ of coroots in $\Xi$ which
is invariant under the $W$-action on $\Xi$ and satisfies the
property that
$\pairwith{\alpha}{\xi}=\pairwith{w\alpha}{w\xi}$ for all $w\in W$, $\alpha\in \Phi$, $\xi\in\Phi^*$. We define the
affine Weyl group $\curlW$ of $G$ relative to $T$ as
$\curlW:=W\ltimes \{ t_\xi | \xi\in\Xi\}$ where $t_\xi:x\mapsto
x+\xi$ is a translation on $\curlV^*$ and the action is given by
$(w_1t_{\xi_1})(w_2t_{\xi_2})=w_1w_2t_{w_2^{-1}\xi_1+\xi_2}$.

The chosen fundamental roots define sets of positive and negative roots which we write as $\posroots$ and $\negroots$ respectively. Put 
\begindm
U^{+}:=\prod_{\alpha\in{\Phi^{+}}}U_{\alpha}
\enddm
and similarly for $U^{-}$. We are now able to define the Iwahori subgroup. This is given as ${\displaystyle
\curlB:=U^{+}(\pi\localringofintegers)T(\localringofintegers)U^-(\localringofintegers).}$
We have the following $\maxideal$-adic Bruhat Decomposition:

\begin{displaymath}G(\localfield)=\disjointunion_{\wtxi\in\curlW}\curlB
g_w\xipi\curlB\end{displaymath} and
\begin{equation}\label{equation bruhat decomposition second part}G(\localringofintegers)=\disjointunion_{w\in W}\curlB g_w\curlB .\end{equation}

(\cite{IwMa} and \cite[p.\ 74]{Iwahori}). Define a length function $\lambda$ on $\curlW$ by
\begin{eqnarray*}q^{\lambda(\wtxi)}&:=&\textrm{card}(\curlB g_w\xipi\curlB/\curlB)\\
&=&\muG(\curlB g_w\xipi\curlB)/\muG(\curlB).  \end{eqnarray*}
The restriction of $\lambda$ to the finite Weyl group agrees with the
usual length function on $W$. For each $\xi\in\Xi$ there is a unique
element $w_\xi\in W$ such that
$\lambda(\wtxi)$, as a function of $w$, attains a minimum
precisely at $w_\xi$ and
\begin{equation}\label{equation first lambda}
\lambda(w\wxitxi)=\lambda(w)+\lambda(\wxitxi)\ \ \mbox{for all}\ \
w\in W
\end{equation}
\cite[p.\ 21]{IwMa}. As noted in \cite[p.\ 704]{Igusa}, it follows that
\begin{equation}\label{equation second lambda}
\lambda(w_\xi t_\xi)=\sum_{\alpha\in
w_\xi^{-1}(\posroots)}\alphapairxi-\lambda(w_\xi).
\end{equation}
Put
\begindm
\begin{array}{lll}
\Xi^+&:=&\{\xi\in\Xi\suchthat\xipi\in G^+\}, \\
\Xiw&:=&\{\xi\in\Xi\suchthat\wxi=w\},\\
\Xiwplus&:=&\Xiplus\cap\Xiw.
\end{array}
\enddm
\begin{proposition}\label{first combinatorial expression proposition}
Put $\alpha_0:=\prod_{\alpha\in\posroots}\alpha$. Then for almost all primes $\maxideal$,
\begindm
\ZcurlGrhoidealps=\sum_{w\in W}q^{-\lambda(w)}\sum_{\xi\in
\wXiwplus}q^{\pairwith{\alpha_0}{\xi}}|\detrho(\xi(\unifparam))|^s\theta(\xi(\unifparam)).
\enddm
\end{proposition}
Note that this is essentially the same as \cite[(5.4)]{duSL}. In their setting, du Sautoy and Lubotzky assume that $\theta$ is (the $\maxideal$-adic absolute value of) a character of $G$, which allows them to express the product $|\detrho(\xi(\unifparam))|^s\theta(\xi(\unifparam))$ in terms of a generator $f$ for the (in their case one-dimensional) character group of $G$.
\begin{proof}
By Theorem~\ref{theorem reduction of integral} we have $\ZnonredGrhops=\ZGrhops$. We may therefore take Definition~\ref{definition reduced integral} as a starting point and follow the line of argument given in \cite[pp. 76-77]{duSL}. The
$\maxideal$-adic Bruhat Decomposition gives
\begin{displaymath}G^+=\disjointunion_{{\wtxi\in\curlW}\atop{\xi\in\Xi^+}}\curlB
g_w\xipi\curlB.\end{displaymath} By construction, $\curlB\subseteq
G(\localringofintegers)$ and $g_w\in G(\localringofintegers)$. It follows that each of
$|\detrho(h)|$ and $\theta(h)$ is constant on the double coset
$\curlB g_w \xi(\unifparam)\curlB$, the latter by Lemma~\ref{theta constant on
double cosets lemma}. In \cite{duSL}, this observation is combined
with (\ref{equation bruhat decomposition second part}),
(\ref{equation first lambda}) and (\ref{equation second lambda}) above to
obtain the desired result. In their argument they use the fact that
each of the maps $\xi\mapsto |\detrhoxipi|$ and
$\xi\mapsto\thetaxipi$ is constant on $W$-orbits in $\Xi$. It
remains to check the latter of these in our setting: We have
\begin{eqnarray*}
\theta((w\xi)(\unifparam))&=&\theta(g_w^{-1}(\xi(\unifparam))g_w)\\
&=&\theta(\xi(\unifparam)),
\end{eqnarray*}
by Lemma~\ref{theta constant on double cosets lemma}.
\end{proof}
Next we need to use information about the weights of the
representation~$\rho|_G$. This will allow us to realize the sets $\wXiwplus$ as subsets of lattice points of a polyhedral cone so that
the results of Section~\ref{reciprocity section} will apply to the weighted sum in Proposition~\ref{first combinatorial expression proposition}.
Along the way, we will define a polyhedral cell complex that will allow us to incorporate the function $\theta$ into
our analysis. Let $\rho_1,\ldots,\rho_r$ be the irreducible
components of $\rho|_G$. For each irreducible component $\rho_i$, let
$\omega_{i1},\ldots,\omega_{in_i}\in\HomTGm$ be the weights of
$\rho_i$ (so $\sum_{i=1}^r n_i=n$) and let $\omega_i\in\HomTGm$ be
the dominant weight of the contragredient representation
$g\mapsto{}^t{\rho_i(g)^{-1}}$. Then there exist
$c_k(j,i)\in\naturalszero$ such that
\begin{eqnarray}\label{weight expression}
\omega_{ij}=\omega_i^{-1}\prod_{k=1}^{l}\alpha_k^{\ckji}
\end{eqnarray}
for each $i\in [r]$ and $j\in [n_i]$ (see, for instance, \cite[31.2, Proposition]{humphreys1}). For convenience, for each
$k\in[n]$ we write $\tilde{\omega}_k$ for the $k^{th}$ weight
in the ordering
$(\omega_{11},\ldots,\omega_{1n_1},\ldots,\omega_{r1},\ldots,\omega_{rn_r})$.
We will assume that the weights are ordered such that
$\tilde{\omega}_k(t)$ is the eigenvalue for the action of $t\in T$ on the
basis element $u_k$ of~$V$ (see Definition~\ref{good representation
definition}). Fix a $\integers$-basis
\begin{equation}\label{Z-basis expression}
\xibasisi_1,\ldots,\xibasisi_{l+d}
\end{equation}
for $\Xi$ and let ${f^*}: \Xi \to \integers^{l+d}$ be the coordinate map relative to this basis. Let $f:\HomTGm\to \integers^{l+d}$ be the coordinate map relative to the dual basis. Note that
$\pairwith{\alpha}{\xi}=\fnew(\alpha).{f^*}(\xi)$ for all $\alpha\in
\HomTGm$, $\xi\in \Xi$, where the latter is the standard inner product on $\reals^{l+d}$. For $\boldv\in\integers^{l+d}$, let $\xibasis^{\boldv}$ denote $\xibasisi_1^{v_1}\ldots \xibasisi_{l+d}^{v_{l+d}}$ (so~$f^*(\xibasis^{\boldv})=\boldv$). 

Igusa showed 
that
\begin{displaymath}
\begin{array}{lll}\wXiw = \{\xi\in\Xi \suchthat \pairwith{\alpha_i}{\xi}\geq0\ \textrm{for all}\ i;\ \pairwith{\alpha_i}{\xi}>0\ \textrm{if}\ \alpha_i\in{w(\negroots)}\}
\end{array}
\end{displaymath}
\cite[pp.\ 702-3]{Igusa}. From this and (\ref{weight expression}) du Sautoy and Lubotzky deduced that

\begin{eqnarray}\label{w-open cone expression}
\wXiwplus = \wXiw\cap \{\xi\in\Xi\suchthat \pairwith{\omega_j^{-1}}{\xi}\geq0\ \ \textrm{for}\ \ j=1,\ldots,r\}.
\end{eqnarray}
\cite[Lemma~5.4]{duSL}. Put\\
\begindm
\begin{array}{llll}
B_{i}^{\geq}&:=\{\bolde\in\reals^{l+d}\suchthat f(\alpha_i).\bolde\geq 0\}& \mbox{for}\ \ i=1,\ldots,l;\\
B_{l+j}^{\geq}&:=\{\bolde\in\reals^{l+d}\suchthat f(\omega_j^{-1}).\bolde\geq 0\}& \mbox{for}\ \ j=1,\ldots,r.
\end{array}
\enddm

We define our polyhedral cone as
\begindm
\curlC=:\bigcap_{i=1}^{l+r}B_i^{\geq}. 
\enddm
For each $i, j\in [n]$ with $i\neq j$, put 
\begindm
H^{\leq}_{ij}:=\{\bolde\in\reals^{l+d}\suchthat \fnew(\weightsi_i\weightsi_j^{-1}).\bolde\leq0\}. 
\enddm
For each $\sigma\in \Sym(n)$
put 
\begindm
\begin{array}{lll}
\Xi(\sigma)&:=&\{\xi\in \Xi\suchthat \xi(\unifparam)\in
T_{\sigma}(\localfield)\},\\
\Xiplus(\sigma)&:=&\Xiplus\cap\Xi(\sigma).
\end{array}
\enddm
(see Definition~\ref{definition T sigma}). Observe that
\begindm
{f^*}\left(\Xi(\sigma)\right)=\bigcap_{1\leq i<j\leq
n}H^{\leq}_{\sigma(i)\sigma(j)}
\enddm
and 
\begindm
\curlC_{I_w}=\curlC\cap \bigcap_{i\in I_w}B_i^>, 
\enddm
with $I_w\subseteq [l]$ as defined in (\ref{equation definition of I_w}). We thus have 
\begindm
{f^*}(\wXiwplus)=\curlC_{I_w}\cap\integerslplusd. 
\enddm
We are now ready to define our polyhedral cell complex $(\curlC,\curlF)$: take $\curlC$ as above and $H_{ij}$ ($1\leq i<j\leq n$) as internal hyperplanes. By Lemma~\ref{local linearity of theta
lemma}, for each $\sigma\in\Sym(n)$ we can find non-negative
integers $m_i(\sigma)$ such that
\begindm
\theta(\xibasis^{\bolde}(\unifparam))=q^{m_1(\sigma)\pairwith{\tilde{\omega}_1}{\xibasis^{\bolde}}+\cdots+m_n(\sigma)\pairwith{\tilde{\omega}_n}{\xibasis^{\bolde}}}
\enddm
for all $\bolde\in {f^*}\left(\Xiplus(\sigma)\right)$, and this
equality holds for almost all primes $\maxideal$. Fix some total ordering on
$\Sym(n)$. To each cell $F$ in the complex $(\curlC, \curlF)$
associate the minimal element $\sigma_F$ of $\Sym(n)$ satisfying
$F\subseteq {f^*}(\Xiplus(\sigma_F))$. Define a function $\gamma:
\curlF\to\integerslplusd$ by
\begin{equation}\label{equation replaceable by constant function}
\gamma(F):=f\left(\prod_{i=1}^n
(\tilde{\omega}_i)^{m_i(\sigma_F)}\right).
\end{equation}
It follows that
$\theta(\xibasis^{\bolde}(\unifparam))=q^{\gamma(F_{\bolde}).\bolde}$
for all $\bolde\in\Xiplus$ and that $\gamma$ is piecewise constant on the
complex (cf. Definition~\ref{replaceable by constant definition}). We are
ready to give our combinatorial expression. By
Proposition~\ref{first combinatorial expression proposition}, we
have
\begin{proposition}\label{second combinatorial expression proposition}
For almost all primes $\maxideal$,
\begin{eqnarray*}
\ZcurlGrhoidealps&=&\sum_{w\in W}q^{-\lambda(w)}
\genfn_{\curlC_{I_w},A,B,\gamma}(q,q^{-s}),
\end{eqnarray*}
where $A:=\fnew(\alpha_0)$, $B:=\fnew(\detrho|_{T})$,
$\gamma$ is the piecewise constant function $\curlF\to\integerslplusd$
defined in (\ref{equation replaceable by constant function}), and $m:=l+d$
(cf. Definition~\ref{generating function definition}).
\end{proposition}
\section{Proof of the main theorem}\label{maintheorem section}
\subsection{Uniformity}\label{uniformity subsection}
Since we will need to apply Lemma~\ref{inversion of a cell lemma}, we first establish the following.
\begin{lemma}\label{non-negative inner product lemma}
For all $w\in W$, $1\neq\xi\in \wXiwplus$ and $F\in\curlF$, we have
$\pairwith{\alpha_0}{\xi}\geq 0$, $\pairwith{\detrho|_{T}}{\xi}>
0$ and $\pairwith{\prod_{i=1}^n
(\tilde{\omega}_i)^{m_i(\sigma_F)}}{\xi}\geq0$. Furthermore, the cone $\curlC$ is pointed.
\end{lemma}
\begin{proof}
The first and third inequalities follow from (\ref{weight expression}) and (\ref{w-open cone expression}). Note that $\detrho|_{T}=\prod_{i=1}^n \tilde{\omega}_i$. Fix $w\in W$ and $\xi\in\wXiwplus$. By (\ref{weight expression}) and (\ref{w-open cone expression}), $\pairwith{\tilde{\omega}_i}{\xi}\geq 0$ for all $i\in [n]$. Thus $\pairwith{\detrho|_{T}}{\xi}\geq 0$. Suppose that $\pairwith{\detrho|_{T}}{\xi}= 0$. Then $\pairwith{\tilde{\omega}_i}{\xi}=0$ for all $i\in [n]$. However, the weights $\tilde{\omega}_i$ generate $\HomTGm$, whence $\xi=1$. The pointedness of $\curlC$ follows directly from the fact that $\pairwith{\detrho|_{T}}{\xi}>
0$ for all $w\in W$, $1\neq\xi\in\wXiwplus$.
\end{proof}
Recall that $\curlF_{I_w}=\{F\in \curlF\suchthat F\subseteq
\curlC_{I_w}\}$. By Proposition~\ref{second combinatorial expression
proposition}, we have
\begin{eqnarray*}
\ZcurlGrhoidealps&=&\sum_{w\in W}q^{-\lambda(w)}
\sum_{F\in\curlF_{I_w}}\genfn_{F,A,B,\gamma}(q,q^{-s}).
\end{eqnarray*}
Thus, by Lemma~\ref{non-negative inner product lemma} and
Lemma~\ref{inversion of a cell lemma}, we have that for almost all
primes $\maxideal$, $\ZcurlGrhoidealps$ is a sum of
rational functions, hence is itself a rational function in $q,
q^{-s}$, depending only on the group $G$ and the representation
$\rho$. This proves part (1) of Theorem~\ref{main theorem}.

\subsection{The Functional Equation}\label{functional equation
subsection} In this section we complete the proof of Theorem~\ref{main theorem} by applying the reciprocity results of Section~\ref{reciprocity section}. Suppose then that the number $r$ of irreducible components of the representation $\rho$ is equal to the dimension $d$ of the maximal central torus $S$. In Section~\ref{bruhat section} we chose an arbitrary basis for $\Xi$ (see (\ref{Z-basis expression})). We now specify this basis. Note that the group $\HomTGm$ is generated by the weights, since $\rho$ is faithful. However, by (\ref{weight expression}), the weights are generated by $\alpha_1,\ldots,\alpha_l$ together with $\omega_1^{-1},\ldots,\omega_r^{-1}$. Since $r=d$ and $\HomTGm$ has rank $l+d$ over $\integers$, this implies that $(\alpha_1,\ldots,\alpha_l,\omega_1^{-1},\ldots,\omega_r^{-1})$ is a $\integers$-basis for $\HomTGm$. We deduce that $\curlC$ is simplicial. Write $\alpha_{l+i}:=\omega^{-1}_i$ for $i=1,\ldots,d$. Let $\xi_1,\ldots,\xi_{l+d}$ be the dual basis to $\alpha_1,\ldots,\alpha_{l+d}$; that is, the elements of $\Xi$ having the property that $\pairwith{\alpha_i}{\xi_j}=\delta_{ij}$ for all $i, j\in [l+d]$. We choose~$(\xi_i)$ and $(\alpha_i)$ as our (ordered) bases for $\Xi$ and $\HomTGm$ respectively and use them to define coordinate maps ${f^*}:\Xi\to\integerslplusd$ and $\fnew:\HomTGm\to\integerslplusd$ as described in Section~\ref{bruhat section}. Note that $\curlC\cap\integers^{l+d}=\naturalszero^{l+d}$, so in fact $\curlC$ is simple. Set $\boldanought=f^*(\xi_{l+1}\ldots\xi_{l+d})=(\underbrace{0,\ldots,0}_l,\underbrace{1,\ldots,1}_d)$. It follows from the definitions and (\ref{weight expression}) that
\begindm
\boldanought\in \integers^{l+d}\cap\curlC\cap\bigcap_{{i,j\in
[n]}\atop{i<j}}H_{ij}
\enddm
and for each $I\subseteq [l]$ we have
\begindm
\curlC_{[l+d]\backslash
I}\cap\integers^{l+d}=\boldanought+\curlC_{[l]\backslash
I}\cap\integers^{l+d}.
\enddm
By Proposition~\ref{second combinatorial expression proposition} and
Corollary~\ref{fn eq corollary} we have
\begindm
\ZcurlGrhoidealps|_{\pgoestopinverse}=(-1)^{l+d}q^{|\Phi^+|\,+(\fnew(\alpha_0)+\loclinfnonZlattice(F_{\boldanought})).\boldanought-\left({\fnew(\det\rho|_T)}.\boldanought\right)
s}\ZcurlGrhoidealps.
\enddm
By construction $\fnew(\alpha_0).\boldanought=\pairwith{\prod_{\alpha\in\posroots}\alpha}{\xi_{l+1}\ldots\xi_{l+d}}=0$. Next, using (\ref{weight expression}) we have ${\fnew(\det\rho|_T)}.\boldanought=\pairwith{\prod_{i=1}^n \tilde{\omega_i}}{\xi_{l+1}\ldots\xi_{l+d}}=n$. (Recall that $n$ is the dimension of the representation $\rho$.) Thus our functional equation becomes
\begindm
\ZcurlGrhoidealps|_{\pgoestopinverse}=(-1)^{l+d}q^{|\Phi^+|\,+\loclinfnonZlattice(F_{\boldanought}).\boldanought-ns}\ZcurlGrhoidealps.
\enddm
By Lemma~\ref{non-negative inner product lemma} we have that $\loclinfnonZlattice(F_{\boldanought})).\boldanought\geq 0$. In particular, if $\nonredG$ is reductive then $\theta(h)=1$ identically on $G=\curlG$ so $\loclinfnonZlattice(F_{\boldanought}))=0$. This
proves part (2) of Theorem~\ref{main theorem}.\\

As promised in Section~\ref{introduction section}, we now explain the error in the proof of \cite[Theorem~A]{duSL}. The authors realize the zeta function in their setting as a generating function over a set $I$ of lattices points of a cone. The existence of what they call a `dominating' dominant weight guarantees that the cone is simplicial. In order to explicitly compute the zeta function, they further require the cone to be simple. To justify this, they state ``Note that, since $\pairwith{\omega_1^{-1}}{\xi}\in\integers$, ${1/m\sum_{j=1}^l (b_j(1)/m_1-c_j)e_j\in\integers}$'' (see \cite[p.\ 82]{duSL}). This inference is not valid, since the second expression is an integer if and only if $1/m_1\pairwith{\omega_1^{-1}}{\xi}$ is an integer (cf. the expression for $\pairwith{\omega_1^{-1}}{\xi}$ on \cite[p.\ 80]{duSL}). This is not true in general, since $\omega_1$ is some dominant weight and there is no control over the integer $m_1$ defined on \cite[p.\ 77]{duSL}. If the `dominating' dominant weight assumption \cite[Assumption~5.5]{duSL} is replaced by the more restrictive assumption that ${|m_1|=1}$ (in particular, ${r=1}$ is sufficient), their cone defining $I$ becomes simple. \\

\section{Counterexamples}\label{counterexample section}

Let $T_d$ denote the $d$-dimensional split torus ($d\geq 2$). We define an infinite family of representations of $T_d$ whose associated local zeta functions do not satisfy functional equations. In these examples, $r=2d-1$, illustrating that the condition $r=d$ in Theorem~\ref{main theorem} cannot be dropped.

\begin{proposition}\label{proposition counterexample}
Let $G=T_d:=\Gmult^d$, where $d\geq 2$. For each integer $k\geq 3$ define a
faithful $\rationals$-rational representation $\rho_{d,k}:G\rightarrow \GL{2d-1}$ by
$\rho_{d,k}(x_1,\ldots,x_d)=\diag(x_1,\ldots,x_d,x_1^k x_d^{-1},\ldots,x_{d-1}^k x_d^{-1})$. For all primes $p$, 
\begindm
Z_{G,\rho_{d,k},p}(s)=\intGplus|\det\rho_{d,k} (g)|^s\mu_G(g) 
\enddm
does not satisfy a functional equation.
\end{proposition}
\begin{proof}
Fix $d$ and $k$, writing $\rho=\rho_{d,k}$ and $G=G(\Qpadic)$. For $\boldy\in\integers^d$ put
\begindm G_{\boldy}=\{\boldx\in G\suchthat v(x_i)=y_i \ \mbox{for}\ i=1,\ldots,d\}.\enddm
We have 
\begindm
G_{\boldnought}=\{{g\in G}\suchthat \rho(g)\in
\GL{2d-1}(\Zp)\}= G(\Zp),
\enddm
and for all
$\boldy\in\integers^{d}$, 
\begindm
G_{\boldy}=G_{\boldnought}.(p^{y_1},\ldots,p^{y_d}). 
\enddm
This implies that
$\mu_{G}(G_{\boldy})=\mu_{G}(G_{\boldnought})=1$. If $\boldx\in G$, put $y_i=v(x_i)$ for $i=1,\ldots,d$.
Put
\begindm
\curlC=\{\boldu\in\realsnonneg^d\suchthat ku_i-u_d\geq 0\ \mbox{for}\ i=1,\ldots,d-1\}.
\enddm
Then $\boldx\in G^+ \iff  \boldy\in\curlC\cap\integers^d$. Setting $X_i=p^{-(k+1)s}$ for $i=1,\ldots,d-1$ and $X_d=p^{(d-2)s}$ we have

\begin{eqnarray*}
Z_{G,\rho,{p}}(s)&=&\intGplus|\detrhog|^s\mu_{G}(g)\\
&=&\sum_{\boldy\in\curlC\cap\integers^d}p^{(-(k+1)(y_1+\cdots+y_{d-1})+(d-2)y_d)s}\mu_{G}(G_{\boldy})\\
&=&\sum_{\boldy\in\curlC\cap\integers^d}X_1^{y_1}\ldots X_d^{y_d}.
\end{eqnarray*}

Let $\{e_1,\ldots,e_d\}$ be the standard basis for $\reals^d$. Put $f_i=e_i$ for $i=1,\ldots,d-1$ and $f_d=ke_d+\sum_{j=1}^{d-1}e_j$. It is straightforward to check that $\curlC=\thespan_{\realsnonneg^{d}}\{f_1,\ldots,f_d\}$. Put $D_0=\{0\}\cup\{je_d+\sum_{i=1}^{d-1}e_i\suchthat j=1,\ldots,k-1\}$. Another routine check shows that
\begindm
\curlC\cap\integers^d=\coprod_{\boldu\in D_0}(\boldu+\thespan_{\naturalszero}\{f_1,\ldots,f_d\}).
\enddm
It follows that
\begin{displaymath}
Z_{G,\rho,{p}}(s)=\frac{1+X_1\ldots X_{d}(1+X_d+\cdots+X_d^{k-2})}{(1-X_1)(1-X_2)\ldots(1-X_{d-1})(1-X_1\ldots X_{d-1}X_d^k)},
\end{displaymath}
hence
\begin{displaymath}
Z_{G,\rho,{p}}(s)|_{{X_i\rightarrow X_i^{-1}}\atop{i=1,\ldots,d}}=(-1)^{d}\frac{X_1\ldots X_d(1+X_d+\cdots+X_d^{k-2}+X_1\ldots X_{d-1}X_d^{k-1})}{(1-X_1)(1-X_2)\ldots(1-X_{d-1})(1-X_1\ldots X_{d-1}X_d^k)}.
\end{displaymath}

Suppose that $Z_{G,\rho,p}(s)$ satisfies a functional equation of the form 
\begindm
Z_{G,\rho,{p}}(s)|_{p\to p^{-1}}=(-1)^m p^{a+bs}Z_{G,\rho,{p}}(s). 
\enddm
Then
\begindm
\begin{array}{ll}
&(-1)^m p^{a+bs}(1+X_1\ldots X_d (1+X_d+\cdots+X_d^{k-2}))\\
=&(-1)^d X_1\ldots X_d (1+X_d+\cdots+X_d^{k-2}+X_1\ldots X_{d-1} X_d^{k-2}).
\end{array}
\enddm

If $d=2$, comparing highest powers of~$p^{-s}$ immediately leads to a contradiction. If ${d>2}$, comparing lowest and next-to-lowest powers of $p^{-s}$ gives $p^{a+bs}=p^{-(d-2)s}$ and ${X_1\ldots X_{d-1}X_d^k=1}$, hence $d+k-1=0$, which is impossible.
\end{proof}

It is easy to see that there are also families of representations of $T_d$ with $r>d$ whose associated zeta functions do satisfy a functional equation. If the associated cone $\curlC$ is simplicial, the existence of a functional equation in fact depends on the configuration of the lattice points in $\curlC$. Recalling the definitions of the sets $D_0$ and $D_1$ used in the proof of Lemma~\ref{inversion of a cell lemma}, we note that the generating function satisfies a functional equation if and only if $D_0$ maps onto $D_1$ under a translation. This is a highly restrictive condition on the cone.\\

\fontsize{8.5}{10.5}\selectfont
\textsc{Department of Mathematics and Applied Mathematics, University of Cape Town, Private Bag, Rondebosch 7701, Cape Town, South Africa}

\textit{E-mail address:} \texttt{mark.berman@uct.ac.za}

\end{document}